\documentclass[a4paper,12pt]{amsart}

\usepackage{a4wide}
\usepackage{tikz}
\usepackage{url}

\newif\ifdetails
\detailstrue
\newcommand{\DETAIL}[1]%
{\ifdetails\par\fbox{\begin{minipage}{0.9\linewidth}\textit{Detail:}
			#1\end{minipage}}\par\fi}
\newcommand{\TODO}[1]%
{\ifdetails\par\fbox{\begin{minipage}{0.9\linewidth}\textbf{TODO:}
			#1\end{minipage}}\par\fi}

\usepackage{makecell}
\usepackage{relsize}

\newtheorem{lem}{Lemma}
\newtheorem{pro}[subsection]{Proposition}
\newtheorem{thm}[subsection]{Theorem}

\theoremstyle{remark}
\newtheorem{exa}{Example}

\newtheorem{defn}[subsection]{Definition}

\usepackage{caption}
\usepackage{subcaption}
\usepackage{float} 
\usepackage[pdftex,a4paper,
citecolor = blue, colorlinks=true,urlcolor=blue]{hyperref}
\newcommand{\old}[1]{{}}

\title{On isomorphism classes of leaf-induced subtrees in topological trees.}

\author{Audace A. V. Dossou-Olory}
\author{Ignatius Boadi}
\address{audace a. v. dossou-olory\\ department of mathematics and applied mathematics\\ university of johannesburg \\ p.o. box 524 \\ auckland park\\johannesburg 2006\\south africa\\ and d{\'e}partement d’hydrologie et de gestion des ressources en eau\\ centre d’excellence d’afrique pour l’eau et l’assainissement\\ institut national de l'eau.
}
\email{audace@aims.ac.za}
\address{ignatius boadi\\african institute for mathematical sciences\\ ghana\\ summerhill estates\\ east legon hills\\ santoe accra\\ ghana}
\email{ignatius@aims.edu.gh}

\begin{document}

\begin{abstract}
A subtree can be induced in a natural way by a subset of leaves of a rooted tree. We study the number of nonisomorphic such subtrees induced by leaves (leaf-induced subtrees) of a rooted tree with no vertex of outdegree 1 (topological tree). We show that only stars and binary caterpillars have the minimum nonisomorphic leaf-induced subtrees among all topological trees with a given number of leaves. We obtain a closed formula and a recursive formula for the families of $d$-ary caterpillars and complete $d$-ary trees, respectively. An asymptotic formula is found for complete $d$-ary trees using polynomial recurrences. We also show that the complete binary tree of height $h>1$ contains precisely $\lfloor 2(1.24602...)^{2^h}\rfloor$ nonisomorphic leaf-induced subtrees.
\end{abstract}

\maketitle

\section{Introduction and preliminary}
Analogous to sampling in statistics, subtrees of trees are very useful when networks are being analyzed. Since networks can be large and also data structures are usually complex and large, subtrees are needed because they are usually manageable to analyze than the whole structure. Also, to determine the reliability of a network when there is vertex or edge failure, one can look at the number of subtrees of the graph which represents the network. The larger the number of subtrees of this graph, the more reliable this network is; see, as examples, \cite{xiao2017trees,  szekely2007binary, szekely2005subtrees}. Furthermore, leaf-induced subtrees are applicable in web analytics, especially in analyzing access patterns of visitors of a particular website \cite{fully}. EvoMiner, an efficient algorithm for frequent subtree mining (FST) of subtrees of pyhologenetic trees is introduced in \cite{evominer}. FST is particularly useful when considering ancestry of current generation which are naturally the leaves of a rooted tree \cite{analysis}.
%

When a vertex of a tree is designated as the root, we have a rooted tree. A topological tree is a rooted tree with no vertex of outdegree one (or degree two, except possibly the root). They are also referred to as series-reduced or homeomorphically irreducible trees \cite{ bergeron1998combinatorial, allman2004mathematical}. By definition, the tree that has only one vertex is a topological tree. A topological tree with each vertex having outdegree at most $d$ is a $d$-ary tree. In a strict or full $d$-ary tree, each vertex has outdegree 0 or $d$. They are applicable mainly in computer science where they are used in designing and building data structures. This leads to efficient data organization, easy access and modification. In \cite{maryapp}, $d$-ary trees are used to develop an efficient solution to the Empirical Cumulative Distribution Function (ECDF) searching problem. The ECDF searching problem arises in multivariate statistics where given $N$ points in $d$-dimensional space, a point $X = \{x_1, x_ 2, \ldots,x_d\}$ dominates another point $Y = \{y_1,y_2,\ldots,y_d\}$ if for all $i \in \{ 1, 2, \ldots,d\}$, $x_i \geq y_i$. The ECDF searching problem primarily deals with finding the ratio of the number of $Ys$ dominated by a given $X$, to $N$, the number of data points \cite{ECDF}. Data structures using $d$-ary trees are designed in \cite{maryapp} which enables efficient ways of conducting ECDF queries.

A caterpillar is a tree whose internal vertices lie on a single path (called the backbone) and all leaves are attached directly to the internal vertices. The caterpillar is found to be an extremal graph when determining the maximum or minimum of various distance-based graph invariants \cite{cater, dadedzi}. For example, Dadedzi et al. found in \cite{dadedzi} that the caterpillar has the greatest distance spectral radius among trees with a given degree sequence. Caterpillars are applicable in layouts of communication and electrical networks where the internal vertices represent major outlets and leaves represent consumers. They are also widely used in chemical graph theory; see, as example, \cite{wagner2018introduction}. They are applied in modeling interactions and in analysis of benzenoid hydrocarbons. Caterpillars are sometimes called Gutman trees, since Ivan Gutman introduced them to chemical graph theory and made extensive use of them in his works; see, as examples, \cite{gutman2000benzenoids, gutman2012introduction}. They are also sometimes called benzenoid trees \cite{elcatapp}.  When one of the internal vertices of a caterpillar is designated as the root, we have a rooted caterpillar. 

We mainly focus on two families of topological trees in this work: 
\begin{enumerate}
	\item $d$-ary caterpillars: A $d$-ary caterpillar is defined as a strict $d$-ary tree which is also a rooted caterpillar. A $d$-ary caterpillar is denoted by $F^d_n$, where $n$ is the number of leaves. In Figure~\ref{2ary}, we show the ternary caterpillar with $7$ leaves, $F^3_7$.
	At each level of a $d$-ary caterpillar, there are $d-1$ leaves and one internal vertex except at the highest level, where there are $d$ leaves and one internal vertex. Thus, the number of leaves of a $d$-ary caterpillar of height h is $1+h(d-1)$. 
	
	\begin{figure}[!h]
		\centering
		\includegraphics[scale=0.55]{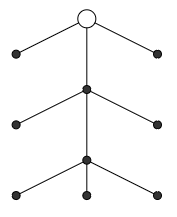}
		\caption{The ternary caterpillar with $7$ leaves, $F^3_7$.}
		\label{2ary}
	\end{figure}
	\item The complete $d$-ary tree of height $h$, denoted by $C^d_h$ is a strict $d$-ary tree which has all leaves at distance $h$ from the root, thus all leaves lie at the same level in the tree. In Figure~\ref{CDdh}, we show the complete ternary tree whose height is two, $C_2^3$. In a complete $d$-ary tree, there are $d$ vertices at level 1.  Each of these vertices are then connected to $d$ vertices at level 2, yielding $d^2$ vertices at level 2. Hence at level $h$, there are $d^h$ vertices. Since all leaves of a complete $d$-ary tree are at level $h$, the number of leaves of $C^d_h$ is $d^h$.	When $h=0$, we have the one-vertex tree and when $h=1$, we have the star with $d$ leaves.
	
	\begin{figure}[h!]
		\centering
		\includegraphics[scale=0.5]{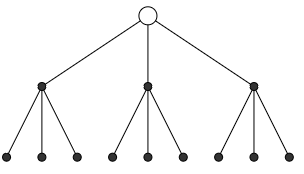}
		\caption{The complete ternary tree with height 2, $C^3_2$.}
		\label{CDdh}
	\end{figure}

\end{enumerate}

Two graphs are
isomorphic if you can label both graphs with the same labels such that any selected vertex has the same neighbors in both graphs.

By a leaf-induced subtree of a rooted tree, we mean a subtree induced by selected leaves of a tree. A subtree can be induced in a natural way by a subset of leaves of a rooted tree. The number of nonisomorphic such leaf-induced subtrees of a given rooted tree is the main subject of this work. The subject of leaf-induced subtrees of rooted trees has been well studied \cite{fibtrees,inducibility,dossou2018inducibility,dossou2019minimum,czabarka2017inducibility,inducibilitytop,furtherresults,inducibilityrooted}.
To form a leaf-induced subtree of a given rooted tree $T$, we consider the power set of the leaves of $T$ excluding the null set. If $S$ is a member of this set, then the subtree induced by the leaves in $S$ is obtained by 
\begin{enumerate}
	\item extracting the smallest subtree of $T$ containing no leaf other than those in $S$. To do this, we first locate the vertex which is the most recent common ancestor of the leaves in $S$. We then select all vertices and edges connecting the leaves in $S$ to the vertex identified as the most recent common ancestor \cite{inducibility,czabarka2017inducibility}, 
	
	\item contracting all vertices of outdegree 1 in the smallest subtree from step 1 \cite{inducibility,czabarka2017inducibility}. Suppose $g$ is a vertex with outdegree 1 in the subtree obtained from step 1 and it has neighbors $f$ and $h$, then we have edges $\{f,g\}$ and $\{g,h\}$; we contract $g$ and create the edge $\{f,h\}$. This is done for all instances of vertices with outdegree 1 to create a leaf-induced subtree.
\end{enumerate}
A single leaf can be selected and used to create a leaf-induced subtree which is the one-vertex tree. By definition, any leaf-induced subtree of a rooted tree is a topological tree. 

\begin{defn}
	We define $N(T)$ as the number of nonisomorphic leaf-induced subtrees of a topological tree $T$.	
\end{defn}

Clearly, there is only one (up to isomorphism) subtree induced by a given subset of leaves of a topological tree. Hence, the total number of leaf-induced subtrees of a topological tree with $n$ leaves is just $2^n -1$ (excluding the empty set). We consider subtrees induced by leaves of a topological tree and select only one representative from every isomorphism class. 
In Example~\ref{leafinduced}, we construct two isomorphic leaf-induced subtrees of the tree in Figure~\ref{rooted1}.
\begin{exa}\label{leafinduced}
	Consider the tree labeled 1 in Figure~\ref{rooted1}. Let $L=\{l_1,l_3,l_4,l_6\}$. We first extract the smallest subtree containing all the vertices and edges lying on the paths connecting the leaves in $L$ to their most recent common ancestor. Hence we get the leaf-induced subtree labeled 2 in Figure~\ref{rooted1}.
	\begin{figure}[h!]
		\centering
		\includegraphics[scale=0.4]{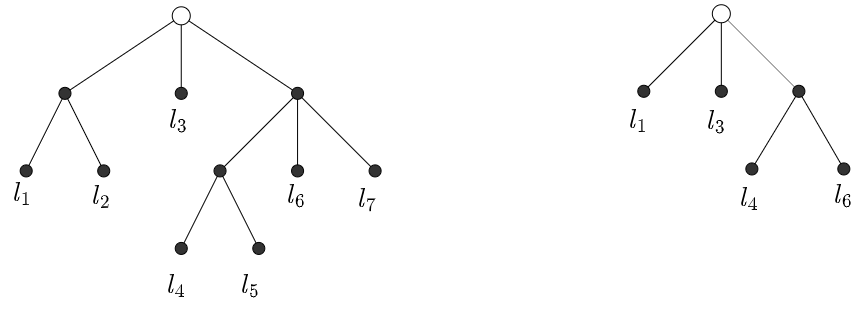}
		\caption{A topological tree and a subtree induced by some of its leaves.}
		\label{rooted1}
	\end{figure}

\end{exa}

\section{main results}

We present formulas for the number of nonisomorphic leaf-induced subtrees of a $d$-ary caterpillar and a complete $d$-ary tree. We also give the minimum number of nonisomorphic leaf-induced subtrees of a topological tree with $n$ leaves together with the characterization of all extremal trees.

All topological trees with up to four leaves are either stars or binary caterpillars except the trees shown in Figure~\ref{nonstar}.
\begin{figure}[h!]
	\centering
	\includegraphics[scale=0.7]{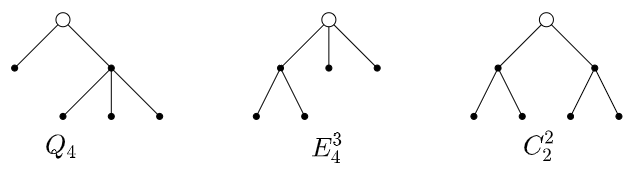}
	\caption{Some topological trees with 4 leaves.}
	\label{nonstar}
\end{figure}

There are 4 leaf-induced subtrees of $C^2_2$ and 5 of each of $Q_4$ and $E^3_4$.
Theorem~\ref{thm1} states the minimum number of leaf-induced subtrees among all topological trees with $n>4$ leaves and also states the corresponding extremal trees.
\begin{thm}\label{thm1} Let $n\geq5$ be an integer. Every $n$-leaf topological tree $T$ has at least $n$ nonisomorphic leaf-induced subtrees, with equality if and only if $T$ is a star or a binary caterpillar. 
\end{thm}
\begin{proof}
	Consider the star $S_n$, with $n$ leaves. Since all leaf-induced subtrees of $S_n$ are themselves stars (see Figure~\ref{stars}), we obtain $n$ nonisomorphic leaf-induced subtrees for $S_n$. Consider the binary caterpillar with $n$ leaves $f_n^2$. All leaf-induced subtrees of $F_n^2$ are also binary caterpillars. Hence we obtain $n$ nonisomorphic leaf-induced subtrees for $F^2_n$ (see Figure~\ref{theorem1}). 
	\begin{figure}[h!]
		\centering
		\includegraphics[scale=0.5]{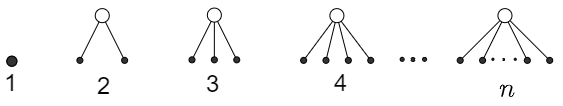}
		\caption{All leaf-induced subtrees of $S_n$.}
		\label{stars}
	\end{figure}
	
	\begin{figure}[h!]
		\centering
		\includegraphics[scale=0.5]{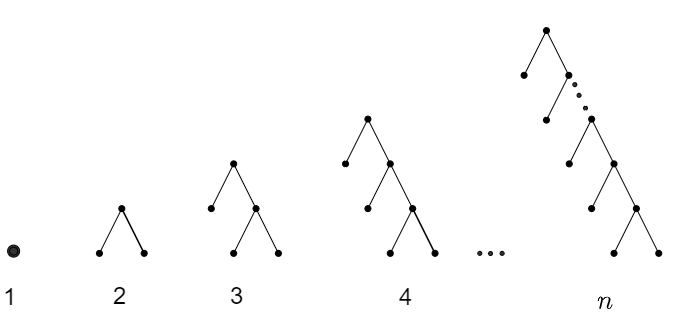}
		\caption{All leaf-induced subtrees of $F^2_n$.}	
		\label{theorem1}
	\end{figure}
	
	Finally, we show that a topological tree $T$ with $n\geq5$ leaves which is neither a star nor a binary caterpillar has at least $n+1$ nonisomorphic leaf-induced subtrees. Since $T$ is neither a star, nor a  binary caterpillar, its root should have outdegree at most $n-1$ and at most $n-2$ leaves attached to it. The key to the proof of this part of the theorem is to show that there exists at least two nonisomorphic leaf-induced subtrees with $k$ leaves for some $k \in \{3, \ldots,n\}$. We consider two possible cases: \begin{enumerate}
		\item If $T$ contains a vertex of outdegree at least three, then the leaf-induced subtrees with three leaves, $S_3$ and $F^2_3$ are both present (see the trees labeled $3$ in Figures~\ref{stars} and \ref{theorem1} respectively). Hence we have at least $n+1$ nonisomorphic leaf-induced subtrees for $T$.
		
		\item If $T$ is a binary tree, then it must be of height at least three since it has at least five leaves. The trees $C_2^2$ (see Figure~\ref{nonstar}) and $F_2^2$ (see the tree labeled 4 in Figure~\ref{theorem1}) each with four leaves, are leaf-induced subtrees of $T$.  Hence we have at least $n+1$ nonisomorphic leaf-induced subtrees for $T$.
	\end{enumerate}
\end{proof}
\section*{2.1 $d$-ary Caterpillars}
\begin{thm}
	Let $d \geq 3$ be a fixed positive integer. For the $d$-ary caterpillar $F_n^d$, we have
	\begin{equation*}
		N(F^d_n) = \dfrac{(d-1)^{\frac{n+d-2}{d-1}}-1}{d-2}.
	\end{equation*}
\end{thm} 
\begin{proof}
	Since $F_n^d$ is a strict $d$-ary tree, it has precisely $(n-1)/(d-1)$ internal vertices which is also its height. Consider the internal vertex at the highest level which has $d$ leaves.
	\begin{enumerate}
		\item We can attach from 2 to $d$ leaves to this vertex to get $d-1$ leaf-induced subtrees. Hence in addition to the one-leaf rooted subtree, we have 1 + $(d-1)$ leaf-induced subtrees.
		\item For each of the $d-1$ leaf induced subtrees from step $1$, we move to the next internal vertex above the vertex considered in step 1 and attach from 1 to $d-1$ leaves of this vertex to give $(d-1)(d-1) = (d-1)^2$ leaf-induced subtrees resulting in $1 +(d-1) + (d-1)^2$ leaf-induced subtrees. 
		\item For each of the $(d-1)^2$ leaf-induced subtrees from step 2, we move to the next vertex and attach from 1 to $d-1$ leaves to give $(d-1)^2(d-1) = (d-1)^3$ leaf -induced subtrees yielding $1 + (d-1) + (d-1)^2 + (d-1)^3$ leaf-induced subtrees in total.
	\end{enumerate}
	This process is repeated for the new subtrees formed up the backbone of the tree until we reach the root. The number of times this operation is done is equal to the height of the tree. We therefore have
	\begin{align*}\nonumber
		N(F^d_n) &= 1 + (d-1) + (d-1)^2  +\cdots+(d-1)^{\frac{n-1}{d-1}}\\\label{formulacat}
		&=  \dfrac{(d-1)^{\frac{n+d-2}{d-1}}-1}{d-2}.
	\end{align*}
\end{proof}

\section*{2.2 Complete $d$-ary Trees}
In this section, we develop a recursive formula and an asymptotic formula for $N(C_h^d)$. Lemma~\ref{lem1} below will be used in deriving the asymptotic formula for $N(C_h^d)$.

\begin{lem}\label{lem1}
	Let $d\geq2$ be a fixed integer and $a_0,a_1,a_2,\ldots,a_d$ be real numbers such that $a_i \geq 0,$ and $a_d \neq 0$. If $(A_n)_{n\geq 0}$ is a polynomial sequence defined recursively by $$A_0> 0, A_n = \sum_{k=0}^{d}a_k.A_{n-1}^k,$$
	then
	\begin{equation}\label{loganlemma}
		\log(A_n) = \dfrac{d^n - 1}{d-1}\log(a_d)+d^n\bigg(\log(A_0)+\sum_{j=0}^{n-1}d^{-1-j}.\log \bigg(1 + \sum_{k=0}^{d-1}\dfrac{a_k}{a_d}.A_j^{k-d}\bigg)\bigg)
	\end{equation}
	for every $n\geq 1.$ Moreover, if $(A_n)_{n\geq 0}$ is an increasing sequence and $\lim\limits_{n\rightarrow\infty}A_n = \infty$, we also obtain
	\begin{equation}
		\small{A_n = (1 + o(1))a_d^{-\frac{1}{d-1}}\exp\bigg(d^n\bigg(\dfrac{\log(a_d)}{d-1}+\log(A_0)+\sum_{j=0}^{\infty}d^{-1-j}\log\bigg(1 + \sum_{k = 0}^{d-1}\dfrac{a_k}{a_d}.A_{n-1}^{k-d}\bigg)\bigg)\bigg)}
	\end{equation}
	as $n\rightarrow \infty$.
\end{lem}
\begin{proof}
	To prove (\ref{loganlemma}), consider the recursion
	\begin{align}\nonumber
		A_n &= a_0 + a_1A_{n-1} + a_2A_{n-1}^2 + \cdots + a_dA_{n-1}^d,\\\nonumber
		\dfrac{A_n}{A^d_{n-1}} &= a_d + \dfrac{a_{d-1}}{A_{n-1}} + \dfrac{a_{d-2}}{A_{n-1}^2} + \cdots+\dfrac{a_0}{A^d_{n-1}},\\\nonumber
		\log\bigg(\dfrac{A_n}{A_{n-1}^d}\bigg)&=\log\bigg(a_d\bigg(1+ \dfrac{a_{d-1}}{a_dA_{n-1}} + \dfrac{a_{d-2}}{a_dA_{n-1}^2} + \cdots+\dfrac{a_0}{a_dA^d_{n-1}}\bigg)\bigg),\\\nonumber
		\log A_n-d\log A_{n-1} &=  \log a_d + \log\bigg(1 + \sum_{k = 0}^{d-1}\dfrac{a_k}{a_d}.A_{n-1}^{k-d}\bigg),\\\label{an}
		\log A_n &= d\log A_{n-1} + \log a_d + \log\bigg(1 + \sum_{k = 0}^{d-1}\dfrac{a_k}{a_d}.A_{n-1}^{k-d}\bigg).
	\end{align}
	From \eqref{an}, we get
	\begin{align}\label{an-1}
		\log A_{n-1} &= d\log A_{n-2} + \log a_d + \log\bigg(1 + \sum_{k = 0}^{d-1}\dfrac{a_k}{a_d}.A_{n-2}^{k-d}\bigg).
	\end{align}
	Substituting (\ref{an-1}) into (\ref{an}), we get
	\begin{align*}
		\log A_n&= d^2\log A_{n-2} + (d+1)\log a_d + d\log\bigg(1 + \sum_{k = 0}^{d-1}\dfrac{a_k}{a_d}.A_{n-2}^{k-d}\bigg)+ \log\bigg(1 + \sum_{k = 0}^{d-1}\dfrac{a_k}{a_d}.A_{n-1}^{k-d}\bigg).
	\end{align*}
	Clearly, we have
	\begin{equation}\label{an-j}
		\log A_{n-j} = d\log A_{n-j-1} + \log a_d + \log\bigg(1 + \sum_{k = 0}^{d-1}\dfrac{a_k}{a_d}.A_{n-j-1}^{k-d}\bigg)
	\end{equation}
	for every $j$ such that $n-j> 0.$ Using (\ref{an-j}) to find expressions for $A_{n-j}, j = 1,2,\ldots,n-1$ and substituting into (\ref{an}), we get
	\begin{align}\nonumber
		\log A_n &=  (d^{n-1} +d^{n-2}+ \cdots + d + 1)\log a_d+d^n \log A_0 +d^{n-1} \log\bigg(1 + \sum_{k = 0}^{d-1}\dfrac{a_k}{a_d}.A_{0}^{k-d}\bigg)\\ \nonumber
		&\quad+\cdots+  d\log\bigg(1 + \sum_{k = 0}^{d-1}\dfrac{a_k}{a_d}.A_{n-2}^{k-d}\bigg)+ \log\bigg(1 + \sum_{k = 0}^{d-1}\dfrac{a_k}{a_d}.A_{n-1}^{k-d}\bigg)\\\label{logan}
		&= \dfrac{d^n - 1}{d-1}\log a_d + d^n\log A_0 + d^n\bigg(\sum_{j = 0}^{n - 1}d^{-1-j}\log\bigg(1 + \sum_{k = 0}^{d-1}\dfrac{a_k}{a_d}.A_{j}^{k-d}\bigg)\bigg).
	\end{align}
	This proves \eqref{loganlemma}.
	Consider the series
	\begin{align*}
		d^n\sum_{j = n}^{\infty} d^{-1-j}\log\bigg(1 + \sum_{k = 0}^{d-1}\dfrac{a_k}{a_d}.A_{j}^{k-d}\bigg).
	\end{align*}
	Note that we have $k<d$. Recall that $a_k/a_d \geq 0$. Since $\sup_{j\geq n} A_j^{k-d} = A_n^{k-d}$, we have
	\begin{align*}
		0\leq d^n\sum_{j = n}^{\infty} d^{-1-j}\log\bigg(1 + \sum_{k = 0}^{d-1}\dfrac{a_k}{a_d}.A_{j}^{k-d}\bigg) &\leq d^n\sum_{j = n}^{\infty}d^{-1-j}.\log\bigg(1 + \sum_{k = 0}^{d-1}\dfrac{a_k}{a_d}.A_{n}^{k-d}\bigg)\\ &= \dfrac{1}{d-1}.\log\bigg(1 + \sum_{k = 0}^{d-1}\dfrac{a_k}{a_d}.A_{n}^{k-d}\bigg)
	\end{align*}
	and 
	\begin{align*}
		\lim\limits_{n\rightarrow\infty}\bigg(\dfrac{1}{d-1}.\log\bigg(1 + \sum_{k = 0}^{d-1}\dfrac{a_k}{a_d}.A_{n}^{k-d}\bigg)\bigg) &= 0.
	\end{align*}
	Hence 
	\begin{equation*}
		d^n\sum_{j = n}^{\infty} d^{-1-j}\log\bigg(1 + \sum_{k = 0}^{d-1}\dfrac{a_k}{a_d}.A_{j}^{k-d}\bigg) = o(1).
	\end{equation*}
	Now, let
	\begin{equation}\label{propref}
		R_n(d) = \sum_{j = n}^{\infty} d^{-1-j}\log\bigg(1 + \sum_{k = 0}^{d-1}\dfrac{a_k}{a_d}.A_{j}^{k-d}\bigg).
	\end{equation}
	We rewrite (\ref{logan}) as 
	\begin{align*}
		\log A_n &= \dfrac{d^n - 1}{d-1}\log a_d + d^n\log A_0 + d^n\bigg(\sum_{j = 0}^{n - 1}d^{-1-j}\log\bigg(1 + \sum_{k = 0}^{d-1}\dfrac{a_k}{a_d}.A_{j}^{k-d}\bigg) + R_n(d) - R_n(d)\bigg)\\
		&= \dfrac{d^n - 1}{d-1}\log a_d + d^n\log A_0 + d^n\bigg(\sum_{j = 0}^{\infty}d^{-1-j}\log\bigg(1 + \sum_{k = 0}^{d-1}\dfrac{a_k}{a_d}.A_{j}^{k-d}\bigg) - R_n(d)\bigg),
	\end{align*}
	which holds for all $n\geq 1$.  Let 
	$$K(d) = \sum_{j = 0}^{\infty}d^{-1-j}\log\bigg(1 + \sum_{k = 0}^{d-1}\dfrac{a_k}{a_d}.A_{j}^{k-d}\bigg).$$
	Then we have 
	\begin{align*}
		\log A_n &= d^n \bigg(\log(A_0) + \dfrac{\log(a_d)}{d-1}+K(d)\bigg)- \dfrac{\log (a_d)}{d-1} + o(1),\\
		A_n &= \exp\bigg(d^n \bigg(\log(A_0) + \dfrac{\log(a_d)}{d-1}+K(d)\bigg) - \dfrac{\log (a_d)}{d-1} + o(1)\bigg)\\
		&= \exp(o(1)).\exp\bigg(\log\bigg(a_d^{\frac{-1}{d-1}}\bigg)\bigg).\exp\bigg(d^n\bigg( \log(A_0) + \dfrac{\log(a_d)}{d-1}+K(d)\bigg)\bigg).
	\end{align*}
	Using Taylor series expansion for $\exp(o(1))$, we have
	\begin{align*}
		&= (1+o(1))a_d^{-\frac{1}{d-1}}\exp\bigg(d^n \bigg(\log (A_0)+\dfrac{\log (a_d)}{d-1}+K(d)\bigg)\bigg) 
	\end{align*}
	as $n\rightarrow \infty$. This completes the proof of the lemma.
\end{proof}
\begin{thm}\label{cdform}
	The number of nonisomorphic leaf-induced subtrees of the complete  $d$-ary tree, $C_h^d$  is given by 
	\begin{equation}\label{cdformula}
		N(C_h^d) = -N(C_{h-1}^d) + \begin{pmatrix}
			d + N(C_{h-1}^d)\\
			d 
		\end{pmatrix}
	\end{equation}
	with $N(C_0^d) = 1$. Furthermore, 
	\begin{equation}\label{2ndpart}
		N(C_h^d) = (1+o(1))(d!)^{\frac{1}{d-1}}\kappa(d)^{d^h}
	\end{equation}
	for some effectively computable constant $\kappa(d)>1$ as $h\rightarrow \infty$. 
\end{thm}

\begin{proof}
Fix $h\geq1$. Call $v$ the root of $C^d_h$ and $A_h$ the set of all nonisomorphic leaf-induced subtrees of $C_h^d$ whose root is $v$ and has root degree $r$ is uniquely determined by a choice of a $r$-element multiset $A_{h-1}$ (since all $r$ branches lie in $A_{h-1}$). So for every $r\in\{2,3,\ldots,d\}$, there are precisely
	\begin{equation*}\label{combrep}
		\begin{pmatrix}
			N(C^d_{h-1}) + r - 1\\r
		\end{pmatrix}
	\end{equation*}
such subtrees (with root degree $r$). Therefore, taking into consideration the subtree consisting only of the single leaf, and using the fact that the elements of $A_{h-1}$ are already counted this way, we get
	\begin{equation*}\label{eqncd}
		N(C_h^d) = 1 + \sum_{r = 2}^{d}\begin{pmatrix} N(C^d_{h-1})  + r - 1\\r \end{pmatrix}
	\end{equation*}
	nonisomorphic leaf-induced subtrees. Let $j = N(C^d_{h-1}) + r - 1$.  We get
	\begin{align*}
		N(C^d_h)&= 1 + \sum_{j = N(C^d_{h-1}) + 1}^{N(C^d_{h-1})  + d - 1}\begin{pmatrix} j\\j + 1 -N(C^d_{h-1}) \end{pmatrix}\\
		&= 1 + \sum_{j = N(C^d_{h-1}) + 1}^{N(C^d_{h-1})  + d - 1}\begin{pmatrix} j\\  N(C^d_{h-1}) - 1\end{pmatrix} \quad \text{since $\begin{pmatrix}
				n\\r
			\end{pmatrix} = \begin{pmatrix}
				n\\n-r
			\end{pmatrix}$}\\
		&= -N(C^d_{h-1})+N(C^d_{h-1}) + 1 + \sum_{j = N(C^d_{h-1}) + 1}^{N(C^d_{h-1})  + d - 1}\begin{pmatrix} j\\ N(C^d_{h-1}) - 1\end{pmatrix}\\
		&= -N(C^d_{h-1}) + \sum_{j = -1+N(C^d_{h-1})}^{N(C^d_{h-1})  + d - 1}\begin{pmatrix} j\\  N(C^d_{h-1}) - 1\end{pmatrix}.
	\end{align*}
	It can be shown by induction on $d$ that
	\begin{equation*} \label{*}
		\sum_{j = -1+N(C^d_{h-1})}^{N(C^d_{h-1})  + d - 1}\begin{pmatrix} j\\  N(C^d_{h-1})- 1 \end{pmatrix} = \begin{pmatrix}
			d + N(C_{h-1}^d)\\N(C_{h-1}^d)
		\end{pmatrix}.
	\end{equation*}
	
%
%
%
%
Hence we have
	\begin{equation*}
		N(C^d_h) = -N(C^d_{h-1}) + \begin{pmatrix}
			d + N(C_{h-1}^d)\\N(C_{h-1}^d)
		\end{pmatrix} = -N(C^d_{h-1})+\begin{pmatrix}
			d + N(C_{h-1}^d)\\d
		\end{pmatrix}.
	\end{equation*}
	Observe that $N(C_h^d)$ increases with $h$ since the nonisomorphic leaf-induced subtrees of $C^d_h$ include all nonisomorphic leaf-induced subtrees of $C^d_{h-1}$. To prove \eqref{2ndpart}, first set $A_h = N(C_h^d)$. We have
	\begin{align*}
		A_h &= -A_{h-1} + \begin{pmatrix}
			d+A_{h-1}\\d
		\end{pmatrix}\\
		&= -A_{h-1}+\dfrac{(A_{h-1} + d)!}{(A_{h-1}!)(d!)}\\
		&= -A_{h-1} + \dfrac{(A_{h-1} + d)(A_{h-1} + d-1)(A_{h-1} + d-2)\cdots(A_{h-1}+2)(A_{h-1}+1)}{d!}\\
		&= -A_{h-1} + \dfrac{A_{h-1}^d +  \bigg(\sum_{i=1}^{d}i\bigg)A_{h-1}^{d-1} +\cdots+ \bigg(\sum_{i = 1}^{d}\prod_{j = 1,j\neq i}^{d}j\bigg)A_{h-1} + d!}{d!}
	\end{align*}
	\begin{align*}
		&= \dfrac{-A_{h-1}d! + A_{h-1}^d +  \bigg(\sum_{i=1}^{d}i\bigg)A_{h-1}^{d-1}+ \cdots+ \bigg(\sum_{i = 2}^{d}\prod_{j = 1,j\neq i}^{d}j\bigg)A_{h-1} +d!A_{h-1}+ d!}{d!}\\
		&=  \dfrac{ A_{h-1}^d +  \bigg(\sum_{i=1}^{d}i\bigg)A_{h-1}^{d-1} +\cdots+ \bigg(\sum_{i = 2}^{d}\prod_{j = 1,j\neq i}^{d}j\bigg)A_{h-1}+ d!}{d!}.
	\end{align*}
	By setting $a_i$ as the coefficients of $A_{h-1}^i$, $i \in \{0,1,2,3,\ldots,d\}$, noting that $A_0 = 1$ and $a_d = 1/d!$ and using the identity derived in \eqref{an}, we have
	\begin{align}\label{logncd}
		\log A_h &= d\log A_{h-1} - \log d! + \log\bigg(1 + d! \sum_{k = 0}^{d-1}a_k A_{h-1}^{k-d}\bigg).
	\end{align}
	
	From Lemma \ref{lem1}, we get 
	\begin{equation*}
		N(C_h^d) = (1+o(1))(d!)^{\frac{1}{d-1}}\kappa(d)^{d^h},
	\end{equation*}
	where 
	\begin{align}\label{kappa(d)}
		\kappa(d) &= \exp\bigg(-\dfrac{\log d!}{d-1}+K(d)\bigg)
	\end{align}
	and
	\begin{align}\label{kd}
		K(d) &= \sum_{j = 0}^{\infty}d^{-1-j}.\log\bigg(1 + d!\sum_{k=0}^{d-1}a_k(N(C^d_j)^{k-d})\bigg).
	\end{align}
	
	From \eqref{logncd}, we get
	\begin{align}\nonumber
		\log A_{h+1} &= d\log A_h - \log d! + \log\bigg(1 + d! \sum_{k = 0}^{d-1}a_k A_{h}^{k-d}\bigg),\\\label{subs}
		\log \bigg(\dfrac{d!A_{h+1}}{A_h^d}\bigg)&=\log\bigg(1 + d! \sum_{k = 0}^{d-1}a_k A_{h}^{k-d}\bigg). 
	\end{align}	
	Recall that $A_h = N(C_h^d)$. Substituting \eqref{subs} into \eqref{kd}, we get
	\begin{align}\label{K(d)}
		K(d)
		&= \sum_{j=0}^{\infty}d^{-1-j}.\log \bigg(d!.\dfrac{N(C^d_{j+1})}{N(C^d_j)^d}\bigg).
	\end{align}
	We have
	$$-\dfrac{\log d!}{d-1}+K(d)=-\dfrac{\log d!}{d-1} + \sum_{j=0}^{\infty}d^{-1-j}.\log \bigg(d!.\dfrac{N(C^d_{j+1})}{N(C^d_j)^d}\bigg).$$
	Since $N(C^d_{j+1})>N(C^d_j)^d$, we obtain
	\begin{align*}
		-\dfrac{\log d!}{d-1}+K(d)&> -\dfrac{\log d!}{d-1}+\sum_{j=0}^{\infty}d^{-1-j}.\log d!\\
		&= -\dfrac{\log d!}{d-1} + \dfrac{\log d!}{d-1}\\
		&=0.
	\end{align*}	
	 Hence $\kappa(d) > 1$.
\end{proof}
The value of $\kappa(d)$ derived in \eqref{kappa(d)} for every $d$ can be numerically evaluated. Table~\ref{Table2} gives values of $\kappa(d)$ for $d\in \{2,3,\ldots,10\}$.

\begin{table}[h!]
	\centering
	\caption{Approximated values of $\kappa(d)$.}
	\begin{tabular}{|c|c|}
		\hline
		$d$& $\kappa(d)$\\
		\hline
		2&1.246020832983661\\
		\hline
		3&1.254860390384554\\
		\hline
		4&1.2189114976086313\\
		\hline
		5&1.1888457507131132\\
		\hline
		6&1.165394603276801\\
		\hline
		7&1.1469724134908297\\
		\hline
		8&1.1322182196849957\\
		\hline
		9&1.1201639471936817\\
		\hline
		10&1.1101387293827483\\
		\hline 
	\end{tabular} 
	\label{Table2}
\end{table}
From Table~\ref{Table2}, when $d=2$, we get
\begin{equation*}
	N(C_h^2) \sim 2(1.246020832983661)^{2^h}
\end{equation*}
as $h\rightarrow\infty$.

Propositions \ref{prop7} and \ref{prop8} below give further results about $\kappa(d)$.
\begin{pro}
	\label{prop7}
	The sequence $\kappa(d)_{d\geq2}$ converges to 1 as $d\rightarrow\infty$.
	\begin{proof}From equations \eqref{kappa(d)} and \eqref{kd}, we have
		\begin{equation}\label{kappa}
			\log (\kappa(d)) = -\dfrac{\log d!}{d-1} + \dfrac{\log d!+\log(d)}{d}+ \sum_{j\geq1}^{}\dfrac{\log d!+\log(N(C^d_{j+1})) -d.\log(N(C^d_j))}{d^{j+1}}.
		\end{equation}
		From \eqref{cdformula}, we can write
		\begin{align} \nonumber
			N(C^d_{j+1}) &= -N(C^d_j) + \begin{pmatrix} 
				d+N(C^d_j)\\d
			\end{pmatrix} = \sum_{k=0}^{d}a_k.N(C^d_j)^k \\ \nonumber
			&\leq \sum_{k=0}^{d}a_k.N(C^d_j)^d\\ \nonumber
			N(C^d_{j+1}) &\leq N(C^d_j)^d\sum_{k=0}^{d}a_k \\
			\dfrac{N(C^d_{j+1})}{N(C^d_j)^d} &\leq \sum_{k=0}^{d}a_k =d. \label{sub}
		\end{align}
		We make a substitution of  \eqref{sub} into \eqref{kappa}. We have
		\begin{align*}
			0\leq \log(\kappa(d)) &\leq -\dfrac{\log d!}{d-1} + \dfrac{\log d!+\log(d)}{d}+ \sum_{j\geq1}^{}\dfrac{\log d!+\log(d)}{d^{j+1}}\\
			&= -\dfrac{\log d!}{d-1} + \dfrac{\log d!+\log(d)}{d}+ (\log d!+\log(d))\sum_{j\geq1}^{}\dfrac{1}{d^{j+1}}\\
			&= -\dfrac{\log d!}{d-1} + \dfrac{\log d!+\log(d)}{d}+ \dfrac{\log d!+\log(d)}{d(d-1)}\\
			&= \dfrac{\log d}{d-1}
		\end{align*}
		It is clear that
		$$\lim\limits_{d\rightarrow\infty}\dfrac{\log d}{d-1} = 0.$$
		Hence 	$\kappa(d)\rightarrow 1$ as $d\rightarrow\infty$.
	\end{proof}
\end{pro}

\begin{pro}
	\label{prop8}
	Given $d\geq2$, there exists a positive integer $H\geq2$ such that
	\begin{equation*}
		N(C^d_h) = \big\lfloor d!^{\frac{1}{d-1}}\kappa(d)^{d^h}\big\rfloor
	\end{equation*}
	
	for every $h\geq H$.
	\begin{proof}	
		From \eqref{propref}, we can show that the relation
		$$0<R_n(d)\leq\dfrac{d^{-n}}{d-1}.\log\bigg(1+\sum_{k=0}^{d-1}\dfrac{a_k}{a_d}.A_n^{k-d}\bigg)$$
		
		 holds for every $d\geq2$ and every $n\geq1$. We deduce the double inequality
		
		\begin{flalign}\label{manipulate}
			d^n\bigg(\sum_{j=0}^{\infty}d^{-1-j}.\log\bigg(1+d!\sum_{k=0}^{d-1}a_k.A_j^{k-d}\bigg)-\dfrac{d^{-n}}{d-1}.d!\sum^{d-1}_{k=0}a_k.A_n^{k-d}\bigg)-\dfrac{d^n-1}{d-1}\log d!&&\\\nonumber
			\leq\log(A_n)\leq d^n\bigg(\sum_{j=0}^{\infty}d^{-1-j}.\log\bigg(1+d!\sum_{k=0}^{d-1}a_k.A_j^{k-d}\bigg)\bigg)-\dfrac{d^n-1}{d-1}\log d!&&
		\end{flalign}
	From \eqref{manipulate}, we obtain
		\begin{flalign}\label{kref}
			d^n\bigg(\sum_{j=0}^{\infty}d^{-1-j}.\log\bigg(1+d!\sum_{k=0}^{d-1}a_k.A_j^{k-d}\bigg)\bigg) - \dfrac{d!}{d-1}A_n^{-1}\sum^{d-1}_{k=0}a_k-\dfrac{d^n}{d-1}\log d!&&\\ \nonumber +\dfrac{1}{d-1}\log d!
			\leq\log(A_n)\leq	d^n\bigg(\sum_{j=0}^{\infty}d^{-1-j}.\log\bigg(1+d!\sum_{k=0}^{d-1}a_k.A_j^{k-d}\bigg)\bigg)-\dfrac{d^n}{d-1}\log d!+\dfrac{1}{d-1}\log d!
		\end{flalign}
	Given that $\sum^{d-1}_{k=0}a_k=\dfrac{dd!-1}{d!}$ and substituting \eqref{K(d)} into \eqref{kref}, we obtain
		\begin{flalign*}	
			d^n\bigg(-\dfrac{\log d!}{d-1}+K(d)\bigg) - \dfrac{d!}{d-1}.A_n^{-1}.\dfrac{dd!-1}{d!} +\log d!^{\frac{1}{d-1}}&&\\\leq\log(A_n)\leq d^n\bigg(-\dfrac{\log d!}{d-1}+K(d)\bigg)+\log d!^{\frac{1}{d-1}}&&\\
			d!^{\frac{1}{d-1}}\kappa(d)^{d^n}.\exp\bigg(-\dfrac{d!}{d-1}.A_n^{-1}.\dfrac{dd!-1}{d!}\bigg)\leq A_n\leq d!^{\frac{1}{d-1}}\kappa(d)^{d^n}
		\end{flalign*}
		Replacing $A_n$ with $N(C_h^d)$ and using the basic inequality $e^t\geq1+t$, we obtain
		\begin{equation}\label{8}
			d!^{\frac{1}{d-1}}\kappa(d)^{d^h} - \dfrac{d!^{\frac{d}{d-1}}}{d-1}.\dfrac{\kappa(d)^{d^h}}{N(C^d_h)}\bigg(d-\dfrac{1}{d!}\bigg)\leq N(C^d_h)\leq d!^{\frac{1}{d-1}}\kappa(d)^{d^h}
		\end{equation}
		We substitute $N(C_h^d)$ in the left-hand side of \eqref{8} with 
		$$d!^{\frac{1}{d-1}}\kappa(d)^{d^h} - \dfrac{d!^{\frac{d}{d-1}}}{d-1}.\dfrac{\kappa(d)^{d^h}}{N(C^d_h)}\bigg(d-\dfrac{1}{d!}\bigg).$$
		Iterating this substitution process a (finite) number of times leads to
		$$ d!^{\frac{1}{d-1}}\kappa(d)^{d^h}-\delta(d)\leq N(C^d_h)\leq d!^{\frac{1}{d-1}}\kappa(d)^{d^h}$$
		for some constant $0<\delta(d)<1.$ An immediate consequence of this double inequality is that
		$$N(C^d_h) = \big\lfloor d!^{\frac{1}{d-1}}\kappa(d)^{d^h}\big\rfloor$$
		holds for all $h\geq H$ for certain $H$.
	\end{proof}
\end{pro}
For $d=2$, 
the previous condition $A_h> (d d!-1) (d-1)$ becomes $A_h > 3$, which means that $N(C^2_h) \geq4$, hence $H=2$. We deduce that
$$N(C^2_h)=\lfloor 2(1.2460208329836625089431529441999359284665241772983812581\ldots)^{2^h} \rfloor$$
for every $h\geq2$.

\end{document}